\documentclass[a4paper,11pt,reqno]{amsart}
\usepackage{graphicx}
\usepackage{a4wide}
\usepackage{amssymb}
\usepackage{amsthm}
\usepackage{mathrsfs}
\usepackage{eucal}
\usepackage{color}
\usepackage[pdftex,colorlinks]{hyperref}

\def\avint{\mathop{\,\rlap{--}\hspace{-0.175cm}\int}\nolimits}

\numberwithin{equation}{section}

\DeclareMathOperator{\Div}{div}
\DeclareMathOperator{\curl}{curl}

\newenvironment{claim}[1]{\par\noindent{\it Claim:}\space#1}{}

\newcommand{\N}{\mathbb N}
\newcommand{\Lvec}{\mathbf L}
\newcommand{\hvec}{\mathbf h}

\newcommand{\Hvec}{\mathbf H}
\newcommand{\Vvec}{\mathbf V}
\newcommand{\uvec}{\mathbf u}
\newcommand{\eps}{\epsilon}
\newcommand{\yvec}{\mathbf y}
\newcommand{\wvec}{\mathbf w}

\newcommand{\be}{\begin{equation}}
\newcommand{\ee}{\end{equation}}
\newcommand{\ba}{\begin{eqnarray}}
\newcommand{\ea}{\end{eqnarray}}

\newcommand{\Om}{\Omega}
\newcommand{\om}{\omega}

\def\dis{\displaystyle}

\newcommand{\mat}[1]{\mbox{\boldmath{$#1$}}}

\newcommand{\zvec}{\mathbf{z}}

\newcommand{\fvec}{\mathbf{f}}
\newcommand{\evec}{\mathbf{e}}

\newcommand{\vvec}{\mathbf{v}}

\newcommand{\ovec}{\mathbf{0}}

\newcommand{\xvec}{\mathbf{x}}

\newcommand{\Avec}{\mathbf{A}}

\newtheorem{theorem}{Theorem}[section]
\newtheorem{proposition}[theorem]{Proposition}
\newtheorem{remark}[theorem]{Remark}

\newtheorem{lemma}[theorem]{Lemma}

\begin{document}

% ========================================================================
% TITLE
% ========================================================================

\title [observability estimate and applications]{Observability inequalities on measurable sets for the Stokes system and applications}

% ========================================================================
% AUTHORS
% ========================================================================

\author{\textsc{Felipe W. Chaves-Silva}}\thanks{Department of Mathematics, Federal University of Pernambuco, CEP 50740-545, Recife,
PE, Brazil. E-mail: {\tt fchaves@dmat.ufpe.br}. F. W. Chaves-Silva  was supported ERC Project No. 320845: Semi Classical Analysis of Partial Differential
Equations, ERC-2012-ADG.
}. 
\author{\textsc{Diego A. Souza}}\thanks{Department of Mathematics, Federal University of Pernambuco, CEP 50740-545, Recife,
PE, Brazil. E-mail: {\tt diego.souza@dmat.ufpe.br}. D. A. Souza was supported
	by the ERC advanced grant 668998 (OCLOC) under the EU's H2020 research program.}

\author{\textsc{Can Zhang}}
\thanks{School of Mathematics and Statistics, Wuhan University, 430072 Wuhan, China; Sorbonne Universit\'es, UPMC Univ Paris 06, CNRS UMR 7598, Laboratoire Jacques-Louis Lions, F-75005 Paris, France. E-mail: {\tt zhangcansx@163.com.}}

\maketitle

% ===========================================================
% ABSTRACT
% ===========================================================

\begin{abstract}
	In this paper, we establish spectral inequalities on measurable sets of positive  Lebesgue measure for 
	the Stokes operator, as well as an observability inequalities on space-time measurable sets of 
	positive measure for non-stationary Stokes system. Furthermore, we provide their applications in 
	the theory of shape optimization and time optimal control problems.
	  
\bigskip

\noindent
\textbf{Keywords\,:}  spectral inequality, observability inequality, Stokes equations, shape optimization problems,
	time optimal control problem.
\vskip 0.25cm

\noindent
\textbf{Mathematics Subject Classification (2010)\,:} 49Q10, 76D07, 76D55, 93B05, 93B07, 93C95. 

\end{abstract}

% ===========================================================
% INTRODUCTION AND MAIN RESULT
% ===========================================================

\section{Introduction and main results}

	Let $T>0$, and let $\Omega\subset \mathbb{R}^N$, $N\geq2$, be a bounded connected open set with a smooth boundary 
	$\partial\Omega$. We will use the notation $Q=\Omega\times(0,T)$,
	$\Sigma=\partial\Omega\times(0,T)$,  and we will denote by $\mat{\nu}=\mat{\nu}(\xvec)$ the outward unit 
	normal vector to $\Om$ at $\xvec\in\partial\Omega$. Throughout the paper spaces of $\mathbb{R}^N$-valued 
	functions, as well as their elements, are represented by boldface letters.

	The present paper deals with an observability inequality on measurable sets of positive measure
	 for the Stokes system
\begin{equation}\label{eq:stokes}
	\left |   
		\begin{array}{lcl}
			\zvec_t - \Delta \zvec  +\nabla q = \ovec 	&  \mbox{in}	&  	Q,  \\
			  \noalign{\smallskip}\dis
			\Div{\zvec}  = 0 				&  \mbox{in}	&    	Q,  \\
			  \noalign{\smallskip}\dis
			\zvec = \ovec 					& \mbox{on} 	& \Sigma, \\
			  \noalign{\smallskip}\dis
			\zvec(\cdot,0) = \zvec_0 			& \mbox{in} 	& \Omega.
		\end{array}
	\right. 
\end{equation}
	System \eqref{eq:stokes} is a linearization of the Navier-Stokes system for a homogeneous viscous incompressible fluid
	(with unit density and unit kinematic viscosity) subject to homogeneous Dirichlet boundary conditions. 
	Here, $\zvec$ is the $\mathbb R^N$-valued velocity field and $q$ stands for the scalar pressure. 	
	
Our motivation to obtain an observability inequality on measurable sets for the Stokes system \eqref{eq:stokes} comes from the well-known fact that observability inequalities are equivalent to controllability properties. In the case we are dealing with, this will be equivalent to the null controllability of system \eqref{eq:stokes} with bounded controls acting on   measurable sets with positive measure, and will have important applications in shape optimization problems and in the study of the bang-bang property for time and norm optimal control problems for system \eqref{eq:stokes} (see 
Section \ref{applications}).

	Observability inequalities for  system \eqref{eq:stokes}  from a cylinder $\omega\times(0,T)$, with 
	$\omega\subset\Omega$ being a non-empty open set, have been proved in different ways by several authors in the past  few years. For instance, in \cite{FI}, the observability inequality for the Stokes system is obtained by means of global Carleman inequalities for parabolic equations with zero Dirichlet boundary conditions (see also \cite{CG} and \cite{FGIP}). Another 
	proof is given in \cite{IPY} by means of Carleman inequalities for parabolic equations with non-homogeneous 
	Dirichlet boundary conditions applied to the system satisfied by the vorticity $curl\, z$.  More recently, in \cite{CL1}, a new proof was established based on a spectral inequality for the eigenfunctions of the Stokes operator.

	Concerning observability inequalities over  general measurable sets in space and time variables, as far as we know, the first result was obtained in \cite{AEWZ} for the heat equation in a bounded and locally star-shaped domain, and later extended in  \cite{EMZ1} and \cite{EMZ2}
	 to the case of  parabolic systems with time-independent  analytic coefficients associated to 
	possibly non self-adjoint elliptic operators and higher order parabolic evolutions with the analytic coefficients 
	depending on space and time variables, when the boundary of the bounded domain in which the equation evolves 
	is global analytic. We also refer the interested 
reader to \cite{AE,PW1,W1} for some earlier and closely related results on this subject.
	
	For the Stokes system, the only result we know is the one in \cite{WZ1}, which gives an observability inequality from a  measurable subset with positive measure in the time variable. In there, the argument is mainly based on the theory of analytic 
	semigroups. In this paper, we extended the result in \cite{WZ1} to the case of observations from sets of positive measure in both time and space variables.

	Before presenting our main results, we first introduce the usual spaces in the context of fluid mechanics:
$$
	\bold{V} =   \{ \yvec \in \Hvec^1_0(\Omega)^N;~\Div \yvec = 0\},
$$
$$
	\bold{H} = \{\yvec \in \Lvec^2(\Omega)^N;~\Div \yvec = 0,~\yvec\cdot\mat{\nu} =0 \mbox{ on }\partial \Omega \}.
$$

Throughout  the paper, the following notation will be used: $B_R(x_0)$ denotes a ball in $\mathbb R^N$ of radius 
	$R>0$ and with  center $x_0\in \mathbb  R^N$; $|\omega|$ is  the  Lebesgue measure of 
	a subset $\omega \subset \Omega$ and  $C(\cdots)$ stands for a positive constant depending only on the parameters 
	within the brackets, and it may vary from line to line in the context.

	Our first result is a $L^1$-observability inequality from measurable sets with positive measure for system \eqref{eq:stokes}.
\begin{theorem}\label{obser}
	Let $B_{4R}(\xvec_0)\subset\Omega$. For any measurable subset $\mathcal{M} \subset B_R(\xvec_0) \times (0,T)$ with positive measure,  there exists a positive constant $C_{obs}=C(N,R,\Omega,\mathcal{M},T)$ 
	such that the observability estimate
\be\label{Obsinequality}
	\|\zvec(T,\cdot)\|_{\bold{H}} \leq C_{obs}\int_{\mathcal{M}}\left| \zvec(\xvec,t)\right|d\xvec dt
\ee
	holds for all $\zvec_0 \in \bold{H}$.
\end{theorem}

\begin{remark}
	When the observation set is $\mathcal M = B_R(\xvec_0) \times (0, T )$, one can see  
	that the observability constant $C_{obs}$ has the form $Ce^{C/T}$ with $C=C(N,\Omega,R)>0$. This is in 
	accordance with the very recent result \cite[Theorem 1.1]{CL1}.
\end{remark}
\begin{remark}\label{rmq:tecnical}
	The above technical assumption imposed on the measurable set $\mathcal M$ is just to simplify the 
	statement of the main result. Without loss of generality, for any measurable set  $\mathcal M \subset \Omega\times(0,T)$ with  positive measure, 
	one  can always assume that 
$$
	\mathcal M \subset B_{R}(\xvec_0)\times(0,T)\;\;\text{with}\;\; \; B_{4R}(\xvec_0)\subset \Omega
$$
	for some $R>0$ and $\xvec_0\in\mathbb R^N$. Otherwise, one may choose a new measurable set 
	$\widetilde{\mathcal M}\subset\mathcal M$ with $|\widetilde{\mathcal M}|\geq c|\mathcal M|$ 
	for some constant $0<c<1$. 
\end{remark}

	The method we shall use to prove Theorem~\ref{obser} relies mainly  on the telescoping series method
	\cite{AEWZ} (which is in part inspired by \cite{Miller2} and \cite{Sei}), the propagation of 
	smallness for real-analytic functions on measurable sets \cite{Vessella} as well as 
	an spectral inequality for Stokes system.
	
	\null
	
	\noindent Let $\{\evec_j\}_{j\geq1}$ be the sequence of eigenfunctions of the Stokes system
\begin{equation}\label{stokesspec}
	\left |   
		\begin{array}{lcl}
			 - \Delta \evec_j  +\nabla p_j = \lambda_j \evec_j  	&  \mbox{in}&  	\Omega,  \\
			  \noalign{\smallskip}\dis
			\Div\evec_j  = 0 						&  \mbox{in}& 	\Omega,  \\
			  \noalign{\smallskip}\dis
			\evec_j = \ovec 							& \mbox{on}& 	\partial \Omega, 
		\end{array}
	\right. 
\end{equation}
	with  the sequence of eigenvalues  $\{\lambda_j\}_{j\geq1}$ satisfying
$$
	0<\lambda_1\leq \lambda_2 \leq \ldots\;\;\;\;\text{and}\;\;\;\;\lim_{j\rightarrow \infty}\lambda_j = +\infty.
$$

	The following inequality is proved in \cite{CL1}.
	
\begin{theorem}\label{5025}\text{\cite[Theorem 3.1]{CL1}}
	For any non-empty open subset $\mathcal O\subset\Omega$, there exists a constant $C=C(N,\Omega,\mathcal O)>0$ 
	such that
\be\label{5021}
	\sum_{\lambda_j \leq \Lambda}a_j^2  = \int_{\Omega} \left | \sum_{\lambda_j \leq \Lambda} a_j \evec_j(\xvec) \right |^2d\xvec 
	\leq Ce^{C\sqrt{\Lambda}} \int_{\mathcal O}\left | \sum_{\lambda_j \leq \Lambda} a_j\evec_j(\xvec) \right |^2 d\xvec,
\ee
	for any sequence of real numbers $\{a_j\}_{j\geq1}\in \ell ^2$ and any positive number 
	$\Lambda$.\footnote{Recall that $\ell^2\triangleq\left\{\{a_j\}_{j\geq1}:\,
	\sum\limits_{j=1}^{+\infty}a_j^2< +\infty\right\} $.}
\end{theorem}

%%%%%%%%%%%%%%%%%%%	
Spectral inequality \eqref{5021} allow us to control  the low frequencies of the Stokes system with a precise estimate on the cost of controllability with respect to the frequency length which, combined with the decay of solutions of \eqref{eq:stokes},  implies the null controllability of Stokes system with $\bold{L}^2$-controls applied to arbitrarily small open sets.

%%%%%%%%%%%%%%%%%%%%

	Our second main result is an extension of the spectral inequality \eqref{5021} from open sets to 
	measurable sets of positive measure.
\begin{theorem}\label{4241}
	Let $B_{4R}(\xvec_0)\subset \Omega$ and let  $\omega \subset B_{R}(\xvec_0)$ be a measurable set with positive measure. Then, there 
	exists a constant $C=C(N,R,\Omega,|\omega|)>0$ such that 
\be \label{specmeasurable}
	\left(\sum_{\lambda_j \leq \Lambda}a_j^2\right)^{1/2}  
	\leq Ce^{C\sqrt{\Lambda}} \int_{\omega}\left| \sum_{\lambda_j \leq \Lambda} a_j\evec_j(\xvec) \right | d\xvec,
\ee
	for all $\Lambda>0$ and  any sequence of real numbers $\{a_j\}_{j\geq1}\in \ell ^2$.
\end{theorem}

\vspace{3pt}

\begin{remark}
	Inequality \eqref{specmeasurable} leads to  a null controllability result for the Stokes system with $\Lvec^\infty$-controls (see Theorem \ref{control_Linfty}).
\end{remark}

%\begin{remark}
%	With a similar explanation to Remark \ref{rmq:tecnical}, one may see that assumption imposed on the 
%	measurable set $\omega$ is just technical and it is used for keeping in accordance with the statement 
%	of Theorem~\ref{obser}. 
%\end{remark}

As we will see below, the proof of Theorem~\ref{4241} strongly depends on quantitative estimates of the interior spatial 
	analyticity for finite sums of eigenfunctions of the Stokes system \eqref{stokesspec}. As far as we know, for the Navier-Stokes equations,  this kind of  interior analyticity  has been first analyzed in \cite{Ka1} and \cite{Ma1}, where the authors consider a nonlinear elliptic system satisfied by the velocity $z$ and the vorticity 
	$curl \,z$  and show  the interior analyticity for the velocity $z$. However, since the boundary condition for the $curl\, z$ is not prescribed, the analyticity up to the boundary cannot be achieved by this method. 
	
	In this paper, in order to establish the spectral inequality \eqref{specmeasurable}, 
	we adapted the arguments in \cite{Ka1} and \cite{Ma1}, and  \cite[Theorem 5]{AEWZ},  to the low frequencies of the Stokes system.

	The paper is organized as follows. 
	In Section~\ref{spectral_measurable}, we shall present the proofs of  Theorems \ref{obser}
	and \ref{4241}. Section~\ref{applications} deals with some applications of main theorems for
	shape optimization and time optimal control problems of Stokes system. 
	Finally, in Appendix~\ref{appendix}, we prove some real-analytic estimates for solutions of the Poisson equation.

% ===========================================================
% THE PROOFS OF MAIN RESULTS
% ===========================================================

\section{Spectral and Observability inequalities}\label{spectral_measurable}

% ===========================================================
% SPECTRAL INEQUALITY ON MEASURABLE SETS
% ===========================================================

\subsection{Spectral inequality on measurable sets}

	This section is devoted to the proof of Theorem~\ref{4241}. Compared to the proof of \cite[Theorem 5]{AEWZ} 
	for the Laplace operator, we here encounter the difficulty due to the pressure  in the Stokes system. To circumvent that, we consider the equation satisfied by the $curl$ of the low frequencies, which is an equation without pressure but with no boundary conditions. %(the $\curl$ of low frequencies do not have homogeneous boundary conditions).
	This allow us recover and quantify the interior real-analytic estimates based on 
	the $curl$ operator.

	We begin with an estimate of the propagation of smallness for real-analytic functions on measurable sets with positive measure, which plays a core ingredient in the proof of Theorem~\ref{4241}.
\begin{lemma}\label{Vessella}
	Assume that $\fvec:B_{2R}(\xvec_0)\subset \mathbb{R}^N\longrightarrow  \mathbb{R}^N$ is real-analytic and verifies
\[
	|\partial_x^\alpha \fvec(\xvec)|\le\frac{M|\alpha|!}{(\rho R)^{|\alpha|}},~\text{for}~\xvec\in B_{2R}(\xvec_0),~\alpha\in\N^N,
\]
	for some $M>0$ and $0<\rho\le 1$. 
	
	For any measurable set  $\omega \subset B_R(\xvec_0)$ with positive measure,  there are positive constants $C=C(R,N,\rho, |\omega|)$ and $\theta=\theta(R,N,\rho, |\omega|)$,  with 
	$\theta\in(0,1)$, such that
\[\label{estmeasurable}
	\|\fvec\|_{\Lvec^\infty(B_R(\xvec_0))}\leq C\left(\int_{\omega}\,|\fvec(\xvec)|\,d\xvec\right)^{\theta}M^{1-\theta}.
\]
\end{lemma}
	\noindent 
	The above-mentioned local observability inequality for real-analytic functions was first established 
	in \cite{Vessella}. The interested reader can also find a simpler proof of Lemma~\ref{Vessella} in \cite[Section 3]{AE}, 
	and a more general extension in \cite[Lemma 2]{EMZ1}.
\vskip 6pt
\begin{proof}[\textbf{Proof of Theorem~\ref{4241}}]
%The following are valid in any dimension $N\geq2$, but we only state the results for $N=3$.
	For each real number $\Lambda>0$ and each sequence $\{a_j\}_{j\geq1}\in \ell^2$, we define 
$$ 
	\uvec_{\Lambda}(\xvec) = \sum_{\lambda_j \leq \Lambda} a_j \evec_j(\xvec),\;\;\xvec\in\Omega,
$$
	and
$$ 
	\vvec_{\Lambda}(\xvec,s) = \sum_{\lambda_j \leq \Lambda} a_j  e^{s\sqrt{\lambda_j}} d \evec_j(\xvec),
	\quad (\xvec,s) \in \Omega \times (-1,1),
$$
	where $d$  denotes the $\curl$ operator.\footnote{In fact, $d$ is the differential which maps $1$-forms into $2$-forms. When a vector field $w$ is identified with a $1$-form, then  $dw$ can be identified  with a ${1\over2}N(N - 1)$-dimensional vector.}

	%We recall that for a scalar function $w$ 
	%or a vector-valued function $\yvec$,  the $\curl$ of $w$ and $\yvec$ is defined by
%$$
%	\curl w=\big(\partial_{x_2}w,-\partial_{x_1}w\big)
%$$
%and
%$$
	%\curl\yvec=\left\{
	%	\begin{array}{lcl}
	%		\partial_{x_1}y_2-\partial_{x_2}y_1&\text{if}&N=2,\\
	%		  \noalign{\smallskip}\dis
	%		\big(\partial_{x_2}y_3-\partial_{x_3}y_2, \partial_{x_3}y_1-\partial_{x_1}y_3,
	%		\partial_{x_1}y_2-\partial_{x_2}y_1\big)&\text{if}&N=3.
	%	\end{array}
	%\right.
%$$	
	%In general, it is a $2$-vector field, which is given in Cartesian coordinates by the $2$-antisymmetric tensor 
%	$d_{i,j}w= {\partial w_{j}\over \partial x_{i}}-{\partial w_{i}\over \partial x_{j}}$, $1\leq i<j\leq N$. 
	
	% In the Riemannian setting, when a vector field $w$ is identified with a $1$-form, then $d(w)$ is identified 
	% with a ${1\over2}N(N - 1)$-dimension.}

	Because
$
	\vvec_{\Lambda}(\cdot,0) = d \uvec_{\Lambda}%,\;\;\xvec\in\Omega
$
	and $\Div_\xvec\uvec_{\Lambda} =0$, we have that
\begin{equation}\label{424-3}
	\Delta_\xvec \uvec_{\Lambda}(\xvec) = d^* \vvec_{\Lambda}(\xvec,0),\;\;\xvec\in\Omega,
\end{equation}
	where $d^*$ is the adjoint  of $d$.   %in $\Omega$.  %\footnote{Particularly,  for $N=2$ or $3$, we used the  identify $-\Delta=\text{curl}\, \text{curl} -\nabla \,\text{div}$.}

	Let us now obtain an estimate of the propagation of smallness for $\uvec_\Lambda$ on  measurable sets with positive measure. According to Lemma~\ref{Vessella}, it is sufficient to     quantify the analytic estimates of higher-order derivatives of $\uvec_\Lambda$.

	 Since $\vvec_{\Lambda}(\cdot,\cdot)$ satisfies 
$$
	-\partial^2_{ss}\vvec_{\Lambda}(\xvec,s)- \Delta_\xvec \vvec_{\Lambda}(\xvec,s)   = 0,  
	\quad (\xvec,s) \in \Omega \times (-1,1),
$$
 	we have that $d^*  \vvec_{\Lambda}$  verifies 
$$
	-\partial^2_{ss} d^*  \vvec_{\Lambda}(\xvec,s)- \Delta_\xvec d^*  \vvec_{\Lambda}(\xvec,s)   = 0,
  	\quad (\xvec,s) \in \Omega \times (-1,1)
$$
	 and, using Lemma~\ref{analyticityDelta} in the appendix with $f\equiv 0$, 
	$d^*\vvec_{\Lambda}$ is real-analytic in $B_{4R}(\xvec_0,0)$ and the following estimate holds
%\begin{align*}
%	\| \partial_\xvec^{\alpha}\partial_s^{\beta} d^* (\vvec_{\Lambda})\|_{\Lvec^{\infty}(B_{2R}(\xvec_0,0))}
%	\leq{(|\alpha|+\beta)!\over (\rho R)^{|\alpha |+\beta +1}}  \|  d^* ( \vvec_{\Lambda})\|_{\Lvec^2(B_{4R} (\xvec_0,0))},
%	\;\;\forall \alpha\in\mathbb N^N,\,\beta\geq0,
%\end{align*}

\begin{align*}
	\| \partial_\xvec^{\alpha}\partial_s^{\beta} d^* \vvec_{\Lambda}\|_{\Lvec^{\infty}(B_{2R}(\xvec_0,0))}
	\leq C{(|\alpha|+\beta)!\over (\rho R)^{|\alpha |+\beta}}  \left(\avint_{B_{4R} (\xvec_0,0)}\!\!\!\!\! 
	|d^*  \vvec_{\Lambda}(\xvec,s)|^2d\xvec ds\right)^{1/2}\!\!\!\!\!,
	\;\;\forall \alpha\in\mathbb N^N,\,\beta\geq0,
\end{align*}
	where the positive constants $\rho$ and $C$ only depend on the dimension $N$.
	
	 Taking $\beta =0$ in the previous estimate, we readily 
	obtain 
%\begin{align}\label{estB1}
%	\|\partial_x^{\alpha} (d^*(\vvec_{\Lambda}(\cdot,0)) \|_{\Lvec^{\infty}(B_{2R}(\xvec_0))}  
%	\leq  (\rho R)^{-|\alpha| -1}|\alpha|!   \|  d^* ( \vvec_{\Lambda}) \|_{\Lvec^2(B_{4R} (\xvec_0,0))},\;\;\forall \alpha\in\mathbb N^N.
%\end{align}

\begin{align}\label{estB1}
	\|\partial_x^{\alpha} d^*\vvec_{\Lambda}(\cdot,0) \|_{\Lvec^{\infty}(B_{2R}(\xvec_0))}  
	\leq  C{|\alpha|!\over (\rho R)^{|\alpha |}}\left(\avint_{B_{4R} (\xvec_0,0)}\!\!\!\!\! 
	|d^*  \vvec_{\Lambda}(\xvec,s)|^2d\xvec ds\right)^{1/2},\;\;\forall \alpha\in\mathbb N^N.
\end{align}

To bound the right-hand side in \eqref{estB1}, we set 
$$
	\wvec_{\Lambda}(\xvec,s ) = \sum_{\lambda_j \leq \Lambda} a_j  e^{s\sqrt{\lambda_j}} \evec_j(\xvec), 
	\;\;(\xvec,s)\in \Omega\times(-1,1)
$$
and then  the following  estimate holds
\[
\begin{alignedat}{2}
	\|  d^*  \vvec_{\Lambda} \|_{\Lvec^2(B_{4R} (\xvec_0,0))}^2 
	%\leq&~\|d^* d \wvec_{\Lambda}\|_{\Lvec^2(\Omega\times(-1,1))}^2\\
	%\noalign{\smallskip}\dis
	\leq & ~C \|\wvec_{\Lambda}\|_{L^2((-1,1); \Hvec^2(\Omega))}^2 \\
	\noalign{\smallskip}\dis
	\leq & ~C \int^1_{-1}\| \Avec \wvec_{\Lambda}(\cdot,s)\|_{\bold{H}}^2 \,ds,
\end{alignedat}
\]
where we have used the fact that there exists $C = C(N, \Omega) >0$ such that
$$
\frac{1}{C} \| \yvec\|_{\Hvec^2(\Omega)} \leq   \| \Avec \yvec\|_{\bold{H}} \leq C \| \yvec\|_{\Hvec^2(\Omega)},
\quad \forall \yvec \in D(\Avec),
$$
with $\Avec$ being  the Stokes operator\footnote{The Stokes operator $\Avec: D(\Avec) \longrightarrow \bold H$ is defined by $\Avec = -P\Delta$, with $D(\Avec) = \big\{\yvec \in \bold V:\; \Avec \yvec \in \bold H\big\}$ and $P : \Lvec^2(\Omega) = \bold H\oplus \bold H ^\perp \longrightarrow \bold H$ is the Leray projection.}

	Since $\{e_j\}_{j\geq 1}$ is an orthonormal basis of $\Hvec$, the last estimate yields
\begin{equation}\label{estSope}
\begin{alignedat}{2}
	\|d^* \vvec_{\Lambda}\|_{\Lvec^2(B_{4R} (\xvec_0,0))}^2 
	%&\leq  C\int^1_{-1}
	%\left\|\sum_{\lambda_j \leq \Lambda} a_j  e^{s\sqrt{\lambda_j}} \lambda_j \evec_j\right\|^2_{\bold{H}}\,ds\\ 
	%\noalign{\smallskip}\dis
	%& \leq C\int^1_{-1}\sum_{\lambda_j \leq \Lambda} a_j^2 \lambda_j^2  e^{2s\sqrt{\lambda_j}}\,ds\\ 
  	& \leq  Ce^{C\sqrt{\Lambda}} \sum_{\lambda_j \leq \Lambda} a_j^2,
\end{alignedat}
\end{equation}
for some $C>0$.

	Therefore, combining \eqref{estB1} and \eqref{estSope}, we have
%\begin{align}\label{B2}
%	\| \partial_x^{\alpha}d^* (\vvec_{\Lambda}(\cdot,0)) \|_{\Lvec^{\infty}(B_{2R}(\xvec_0))}
%	\leq C{|\alpha|!\over (\rho R)^{|\alpha |+1}} e^{C\sqrt{\Lambda}} \left(\sum_{\lambda_j \leq \Lambda} a_j^2\right)^{1/2},\;\;\forall \alpha\in\mathbb N^N,
%\end{align}
\begin{align}\label{B2}
	\| \partial_x^{\alpha}d^* \vvec_{\Lambda}(\cdot,0) \|_{\Lvec^{\infty}(B_{2R}(\xvec_0))}
	\leq C{|\alpha|!\over (\rho R)^{|\alpha |}} e^{C\sqrt{\Lambda}} \left(\sum_{\lambda_j \leq \Lambda} a_j^2\right)^{1/2},\;\;\forall \alpha\in\mathbb N^N,
\end{align}
	\noindent where $C=C(N,\Omega)$.
	
	Since $\uvec_\Lambda$ solves the Poisson equation \eqref{424-3}, we have that $\uvec_\Lambda$ is real-analytic whenever the exterior force $d^*\vvec_{\Lambda}(\cdot,0)$ is real-analytic. Now, thanks to
	\eqref{B2}, we can apply again Lemma~\ref{analyticityDelta} to obtain that
%\begin{align}
%	\| \partial_x^{\alpha}\uvec_{\Lambda}\|_{\Lvec^{\infty}(B_{R}(\xvec_0))}\leq
%	(R\tilde\rho)^{-|\alpha| -1}|\alpha|! \left(\|  \uvec_{\Lambda}\|_{\Lvec^2(B_{2R}(\xvec_0))} +
%	  C\rho^{-1} e^{2\sqrt{\Lambda}} \left(\sum_{\lambda_j \leq \Lambda} a_j^2\right)^{1/2} \right),\;\;\forall \alpha\in\mathbb N^N,
%\nonumber
%\end{align}

\[
	\| \partial_x^{\alpha}\uvec_{\Lambda}\|_{\Lvec^{\infty}(B_{R}(\xvec_0))}\leq
	(R\tilde\rho)^{-|\alpha| -1}|\alpha|! \left(\|  \uvec_{\Lambda}\|_{\Lvec^2(B_{2R}(\xvec_0))} +
	  Ce^{C\sqrt{\Lambda}} \left(\sum_{\lambda_j \leq \Lambda} a_j^2\right)^{1/2} \right),\;\;\forall \alpha\in\mathbb N^N,
\]
for some constant $\tilde\rho>0$.

	Noticing that 
\[
\|\uvec_{\Lambda} \|_{\Lvec^2(B_{2R} (\xvec_0))}^2\leq \|\uvec_{\Lambda} \|_{\Hvec}^2 =  \sum_{\lambda_j \leq \Lambda} a_j^2,
\]
	one can see that
\begin{align}\label{estA6}
	\|\partial_x^{\alpha} \uvec_{\Lambda}\|_{\Lvec^{\infty}(B_{R}(\xvec_0))}
	\leq{|\alpha|!\over(\rho R)^{|\alpha|}}e^{K\sqrt{\Lambda}}\left(\sum_{\lambda_j \leq \Lambda} a_j^2\right)^{1/2},\;\;\forall \alpha\in\mathbb N^N,
\end{align}	
	where $\rho$ and $K$ are positive constants independent of $\Lambda$.

	Using \eqref{estA6} and Lemma \ref{Vessella}, applied to the real-analytic function $\uvec_\Lambda$, we obtain the estimate
\begin{equation}\label{estmeasurable1}
	\|\uvec_{\Lambda}\|_{\Lvec^\infty(B_{R}(\xvec_0))}
	\leq C\left(\int_{\omega}\,|\uvec_{\Lambda}(\xvec)|\,d\xvec\right)^{\theta}
	\left(e^{K\sqrt{\Lambda}}\left(\sum_{\lambda_j \leq \Lambda} a_j^2\right)^{1/2} \right)^{1-\theta}
\end{equation}
	 for some constants $C=C(N,R,\Omega,|\omega|)>0$ and $\theta=\theta(N,R,\Omega,|\omega|)\in(0,1)$.

	On the other hand, by the spectral inequality given in Theorem~\ref{5025}, there exists
	 $C=C(\Omega,R,N)$ such that 
\[
	\left(\sum_{\lambda_j \leq \Lambda} a_j^2\right)^{1/2}
	%\leq Ce^{C\sqrt{\Lambda}} \|\uvec_{\Lambda}\|_{\Lvec^2(B_{R}(\xvec_0))}
	\leq Ce^{C\sqrt{\Lambda}}\|\uvec_{\Lambda}\|_{\Lvec^\infty(B_{R}(\xvec_0))}.
\]
	
	The above inequality and \eqref{estmeasurable1} then leads to
\[
	\left(\sum_{\lambda_j \leq \Lambda} a_j^2\right)^{1/2} 
	\leq Ce^{C\sqrt{\Lambda}} \left(\int_{\omega}\,|\uvec_{\Lambda}(\xvec)|\,d\xvec\right)^{\theta} 
	\left(\sum_{\lambda_j \leq \Lambda} a_j^2\right)^{(1-\theta)/2},
\]
	which give us the desired observability inequality
\[
	\left(\sum_{\lambda_j \leq \Lambda} a_j^2\right)^{1/2} 
	\leq  Ce^{C\sqrt{\Lambda}} \int_{\omega}\,|\uvec_{\Lambda}(\xvec)|\,d\xvec.
\]
\end{proof}

% ===========================================================
% OBSERVABILITY INEQUALITY ON MEASURABLE SETS IN SPACE-TIME VARIABLES
% ===========================================================

\subsection{Observability inequality on measurable sets in space-time variables}

	This Section is devoted to the proof of Theorem~\ref{obser}. %While we are going to use similar ideas from \cite{AEWZ} 
	%and \cite{WZ1}, we here simplify the arguments for the proof of \cite[Theorem 1]{AEWZ} in the case of 
	%heat equations. 
	
	We begin with an interpolation estimate for the solutions of the Stokes system, which will be estimate a consequence of the spectral inequality given in Theorem~\ref{4241} and the exponential decay of solutions 
	of the Stokes system, and can be seen as a quantitative estimate of the strong uniqueness of
	solutions. We refer the reader  to \cite{AEWZ,EMZ1, WZ1} for closely related results concerning  the strong 
	unique continuation property for general parabolic equations.
\begin{proposition}\label{estinterpo}
	Let $B_{4R}(\xvec_0)\subset \Omega$ and let  $\omega \subset B_{R}(\xvec_0)$ be a measurable set with positive measure.
	Then, there exists $C =C(\Omega, |\omega|)>0$ such that 
$$
	\| \zvec(\cdot,t)\|_{\bold{H}}\leq \left(  C e^{\frac{C}{t-s}}\| \zvec(\cdot,t)\|_{\Lvec^1(\omega)}\right)^{1/2}
	\| \zvec(\cdot,s)\|_{\bold{H}}^{1/2},
	\;\;\;\forall \zvec_0 \in \bold{H},
$$
where $0\leq s<t\leq T$ and  $\zvec$ is the solution of \eqref{eq:stokes}  associated to $\zvec_0$.
\end{proposition}
\begin{proof}
	 %By a translation in time, i
	It suffices to prove the estimate in the case  $s=0$. 
	 
	 For any $\Lambda>0$, we set 
$$
	\Hvec_\Lambda\triangleq\text{span}\big\{\evec_j; \lambda_j\leq \Lambda\big\}.
$$
	Given $\zvec_0 \in\bold{H}$, the solution $\zvec$ of \eqref{eq:stokes} can be split into $\zvec= \zvec_{\Lambda}+\zvec_{\Lambda}^{\perp}$,
	where $\zvec_{\Lambda}$ and $\zvec_{\Lambda}^{\perp}$ are the solutions of 
	\eqref{eq:stokes} (together with some pressures) associated to $\zvec_{0,\Lambda}\in \Hvec_\Lambda$ and $ \zvec_{0,\Lambda}^{\perp} \in 
	\Hvec_\Lambda^{\perp}$\footnote{$\Hvec_\Lambda^\perp=\text{span}\big\{\evec_j; \lambda_j> \Lambda\big\}$.}, 
	$\zvec_0 =\zvec_{0,\Lambda}+\zvec_{0,\Lambda}^{\perp}$, respectively.  Moreover, one has
\be\label{DFT0}
	\zvec_{\Lambda}(\cdot,t) \in \Hvec_\Lambda \;\;  \mbox{and} \;\;  
	\|\zvec_{\Lambda}^{\perp}(\cdot,t)\|_{\bold{H}} \leq e^{-\Lambda t}\|\zvec_0 \|_{\bold{H}},
\ee
	for every $t>0$. 

	From \eqref{specmeasurable} and  \eqref{DFT0}, for each $t>0$ we have 
\[
\begin{alignedat}{2}
	\|\zvec(\cdot,t) \|_{\bold{H}} 
	&\leq \|\zvec_{\Lambda}(\cdot,t) \|_{\bold{H}} + \|\zvec_{\Lambda}^{\perp}(\cdot,t) \|_{\bold{H}}\\
	&\leq Ce^{C\sqrt{\Lambda}}\|\zvec_{\Lambda}(\cdot,t) \|_{\Lvec^1(\omega)}  
	+e^{-\Lambda t}\|\zvec_0 \|_{\bold{H}} \\
	& \leq Ce^{C\sqrt{\Lambda}} \left( \|\zvec(\cdot,t) \|_{\Lvec^1(\omega)}
	+\|\zvec_{\Lambda}^{\perp}(\cdot,t)\|_{\bold{H}} \right) +e^{-\Lambda t}\|\zvec_0\|_{\bold{H}}\\
	&\leq Ce^{C\sqrt{\Lambda}} \left( \|\zvec(\cdot,t) \|_{\Lvec^1(\omega)}+e^{-\Lambda t}\|\zvec_0\|_{\bold{H}}\right)
	+e^{-\Lambda t}\|\zvec_0 \|_{\bold{H}}\\
	& \leq C_1e^{C_1\sqrt{\Lambda} - \frac{\Lambda}{2}t}
	\left( e^{\frac{\Lambda}{2} t}\|\zvec(\cdot,t)\|_{\Lvec^1(\omega)}
	+e^{-\frac{\Lambda}{2}t}\|\zvec_0 \|_{\bold{H}}\right)\\
	&\leq C_2e^{\frac{C_2}{t}} \|\zvec(t) \|_{\Lvec^1(\omega)}^{1/2}\|\zvec_0\|_{\bold{H}}^{1/2},\nonumber
\end{alignedat}
\]
	where in the last inequality we have used  that 
$$
	C_1\sqrt{\Lambda} -\frac{t \Lambda}{2}\leq \frac{C_1^2}{2t},\;\;\text{ for any}\;\; \Lambda>0
$$
	 and the following lemma:
\begin{lemma}[\cite{Rob}]\label{LemmaRob}
Let $C_1$, $C_2$ be positive and $M_0$, $M_1$ and $M_2$ be nonnegative. Assume there exist $C_3>0$ such that $M_0 \leq C_3M_1$ and $\delta_0>0$  such that 
$$
M_0 \leq e^{-C_1\delta}M_1 +e^{C_2\delta}M_2 
$$   
for every $\delta \geq \delta_0$. Then, there exits $C_0$ such that 
$$
M_0 \leq C_0M_1^{C_2/(C_1+C_2)}M_2^{C_1/(C_1+C_2)}.
$$ 
\end{lemma}

\end{proof}

	For the proof of Theorem~\ref{obser}, we will use the following result 
	concerning the property of Lebesgue density point for a measurable set in $\mathbb R$.

\begin{lemma}[\cite{PW1}, Proposition $2.1$] \label{diadicdecomposition}
	Let $E$ be a measurable set in $(0,T)$ with positive measure and  let  $l$ be  a density point of $E$. 
	Then, for each $\mu>1$, there is $l_1=l_1(\mu,E)$ in $(l, T)$ such that the sequence 
	$\{l_{m}\}_{m\geq1}$ defined as
\[
	l_{m+1}= l+\mu^{-m}\left(l_1-l\right),\ m=1, 2,\dots
\]
	satisfies
\begin{equation}\label{estimateintervals}
	|E\cap (l_{m+1}, l_m)|\ge \frac 13 \left(l_m-l_{m+1}\right),\ \forall m\ge 1.
\end{equation}
\end{lemma}

\begin{proof}[\textbf{Proof of Theorem~\ref{obser}}]
	For each $t\in (0,T)$, let us  define the slice 
\[
	\mathcal{M}_t=\{x\in \Omega : (x,t)\in\mathcal{M}\}
\]
	and
\[
	E=\left\{t\in (0,T); |\mathcal{M}_t |\ge {|\mathcal{M}|\over2T}\right\}.
\]
	From Fubini's Theorem, it follows that $\mathcal{M}_t\subset\Omega$ is measurable for a.e. $t\in (0,T)$, 
	$E$ is measurable in $(0,T)$ and 
\[
	|E|\ge {|\mathcal{M}|\over2|B_R(\xvec_0)|}
\quad\hbox{and}\quad\chi_E(t)\chi_{\mathcal{M}_t}(\xvec)\le \chi_{\mathcal{M}}(\xvec,t),\ \text{in}\ \Omega\times (0,T).
\]

For a.e. $t\in E$, we apply Proposition~\ref{estinterpo} to  $\mathcal M_t$ to find a constant $C=C(\Omega,R, |\mathcal{M}|/\left(T|B_R(\xvec_0)|\right))$ such that
\begin{equation}\label{estinterpo2}
	\|\zvec(\cdot,t)\|_{\bold H}\le \left(Ce^{\frac C{t-s}}\|\zvec(\cdot,t)\|_{\Lvec^1(\mathcal{M}_t)}\right)^{1/2}
	\|\zvec(\cdot,s)\|_{\bold{H}}^{1/2},
\end{equation}
	for $0\le s<t$.  
	
	Let  $l$ be any density point in $E$. For $\mu>1$ (to be chosen later), 
	we denote by $\{l_m\}_{m\geq1}$ the sequence associated to $l$ and $\mu$ as in
	Lemma \ref{diadicdecomposition}. For each $m\geq1$, we set 
$$
	\tau_m=l_{m+1}+{\left(l_m-l_{m+1}\right)\over6}
$$ 
	hence,
\[
	|E\cap (\tau_m,l_m)|=|E\cap (l_{m+1},l_m)|-|E\cap (l_{m+1},\tau_m)|\ge \frac{(l_m-l_{m+1})}{6}.
\]
Taking $s=l_{m+1}$ in \eqref{estinterpo2}, we get
\begin{equation}\label{estinterpo3}
	\|\zvec(\cdot,t)\|_{\bold H}\le 
	\left(Ce^{\frac C{l_m-l_{m+1}}}\|\zvec(\cdot,t)\|_{\Lvec^1(\mathcal{M}_t)}\right)^{1/2}
	\|\zvec(\cdot,l_{m+1})\|_{\bold{H}}^{1/2}, \ \ \text{for a.e.}\;\;t\in E\cap(\tau_m,l_m).
\end{equation}
%	Integrating \eqref{estinterpo3} with respect to $t$ over $E\cap (\tau_m,l_m)$ and using the 
%	Cauchy-Schwarz inequality, we obtain
	Integrating \eqref{estinterpo3} with respect to $t$ over $E\cap (\tau_m,l_m)$, we obtain
\[
	\|\zvec(\cdot,l_m)\|_{\bold H} \le 
	\left( C e^{\frac{C}{l_m-l_{m+1}}} \int_{l_{m+1}}^{l_m}\chi_E(t)
	\|\zvec(\cdot,t)\|_{\Lvec^1(\mathcal{M}_t)}\,dt\right)^{1/2}\|\zvec(l_{m+1})\|_{\bold H}^{1/2},
\]
	which implies that 
\[
	\|\zvec(\cdot,l_m)\|_{\bold H} \leq \epsilon \|\zvec(\cdot,l_{m+1})\|_{\bold H} 
	+\epsilon^{-1} C e^{\frac{C}{l_m-l_{m+1}}}\int_{l_{m+1}}^{l_m}\chi_E(t)\|\zvec(\cdot,t)\|_{\Lvec^1(\mathcal{M}_t)},
\]
	for any $\epsilon >0$. 

	Taking  $\epsilon=  e^{-\frac{1}{2(l_m-l_{m+1})}}$ in the above inequality, we have 
\begin{equation}\label{EstinterpoA0}
	e^{-\frac{C+{1\over2}}{l_m-l_{m+1}}} \|\zvec(\cdot,l_m)\|_{\bold H} 
	- e^{-\frac{C+1}{l_m-l_{m+1}}}  \|\zvec(\cdot,l_{m+1})\|_{\bold H}  
	\leq  C  \int_{l_{m+1}}^{l_m}\chi_E(t) \|\zvec(\cdot,t)\|_{\Lvec^1(\mathcal{M}_t)}\,dt.
\end{equation} 
	Finally, choosing $\mu=\frac{2(C+1)}{2C+1}$, where $C$ is any constant for which inequality \eqref{EstinterpoA0} holds, we readly obtain % implies
 \begin{align}\label{E:1}
	e^{-\frac{C+{1\over2}}{l_m-l_{m+1}}}\|\zvec(\cdot,l_m)\|_{\bold H}&- e^{-\frac{C+{1\over2}}{l_{m+1}-l_{m+2}}}
	\|\zvec(\cdot,l_{m+1)}\|_{\bold H} \nonumber \\
	& \le \  C\int_{l_{m+1}}^{l_{m}}\chi_E(t) \|\zvec(\cdot,t)\|_{\Lvec^1(\mathcal M_t)}\,dt, \ \ \forall m\geq1,
\end{align} 
because $\mu( l_{m+1} -l_{m+2}) =  l_{m} -l_{m+1},\;\;\text{for all}\;\;m\geq1.$
	
	Finally, adding the telescoping series in \eqref{E:1} from $m=1$ to $+\infty$, we get the observability inequality
\begin{equation*}
\|\zvec(\cdot,T)\|_{\bold H}\le C\int_{\mathcal{M}\cap(\Omega\times [l,l_1])}|\zvec(\xvec,t)|\,d\xvec dt,
\end{equation*}
with some constant $C = C(N,R,\Omega, \mathcal{M}, T)>0$. %This implies \eqref{Obsinequality} and completes the proof.
 \end{proof}

% ===========================================================
% APPLICATIONS
% ===========================================================
\bigskip
\section{Applications}\label{applications}

% ===========================================================
% SHAPE OPTIMIZATION PROBLEMS
% ===========================================================

\subsection{Shape optimization problems} 
	As a direct and interesting application of Theorem~\ref{4241}, we analyze the following shape 
	optimization problem formulated in  \cite{PTZ}.

	Let $\{\beta_j^{\nu}\}_{j\in\mathbb N}$ be a sequence of independent real random variables on 
	a probability space $(\mbox{X},\mathcal{F},\mathbb P)$ having mean equal to $0$, variance equal 
	to $1$, and a super exponential decay (for instance, independent Gaussian or Bernoulli random 
	variables, see \cite[Assumption (3.1)]{BT} for more details). For every $\nu\in\mbox{X}$, the 
	solution of \eqref{eq:stokes} corresponding to the initial datum 
\begin{equation}\label{EspecData}
	\zvec^\nu_0=\sum_{j\geq1}\beta_j^{\nu}a_j\evec_j,\;\;\text{with}\;\;\{a_j\}_{j\geq1}\in\ell^2,
\end{equation}
	is given by 
\begin{equation}\label{EspecSol}
	\zvec^{\nu}(\cdot,t)=\sum_{j\geq1}\beta_j^{\nu}a_je^{-t\lambda_j}\evec_j.
\end{equation}
	Given $L \in (0,1)$, we define the set of admissible designs
\[
	\mathcal U_L=\Big\{\chi_\omega\in L^\infty(\Omega;\{0,1\}):
	\omega\subset\Omega \;\text{is a measurable subset of measure}\;\;|\omega|=L|\Omega|\Big\}.
\]
	For each $\chi_\omega\in\mathcal U_L$, we then define the randomized observability constant by
\[
 	C_{T,rand}(\chi_\omega)=\inf_{||\zvec^{\nu}(T)||=1}\mathbb{E} \int_0^T\int_\omega |\zvec^{\nu}(x,t)|^2d\xvec dt,
\]

Using \eqref{EspecSol}, the properties of random variables $\beta_j^{\nu}$,  and the change of variable $b_j=a_je^{-T\lambda_j}$, we deduce that

\[
 	C_{T,rand}(\chi_\omega)=\inf_{\sum^{\infty}_{j=1}|b_j|^2=1}\mathbb{E} \int_0^T\int_\omega
 	\left|\sum_{j\geq1}\beta_j^{\nu}b_je^{t\lambda_j}\evec_j(\xvec)\right|^2\,d\xvec dt,
\]
	where $\mathbb E$ is the expectation over the space $\mathbb X$ with respect to the probability measure $\mathbb P$. 		
	
	From Fubini's theorem and the independence of the random variables $\{\beta_j^{\nu}\}_{j\in\mathbb N}$,
	a simple computation gives  
$$
	C_{T,rand}(\chi_\omega)=\inf _{j\geq1}\;\;\frac{e^{2T\lambda_j}-1}{2\lambda_j}\int_\omega |\evec_j(\xvec)|^2\,d\xvec.
$$

	We now consider the optimal design problem of maximizing the randomized observability constant  
	$C_{T,rand}(\chi_\omega)$ over the set of admissible designs $\mathcal U_L$. In other words, we study the problem
\begin{equation}\label{OPTDP}
	(P^T):\;\;\;\;\;\;\sup_{\chi_\omega\in\mathcal U_L}C_{T,rand}(\chi_\omega)=\sup_{\chi_\omega\in\mathcal U_L}\;
	\inf _{j\geq1}\;\; \frac{e^{2T\lambda_j}-1}{2\lambda_j}  \int_\omega |\evec_j(\xvec)|^2\,d\xvec.
\end{equation}
	The optimal shape design problem \eqref{OPTDP} models the best sensor shape and location problem for the control of the
	Stokes system \eqref{eq:stokes}.

We have the following result:
\begin{theorem} \label{TheoOpt1}
	The problem $(P^T)$ has a unique solution. 
\end{theorem}
\begin{proof}
We only have to check the following two conditions:

\null

{\item
\begin{itemize}
\item[i.]
If there exists $E \subset \Omega$ of positive Lebesgue measure, an integer $m\in\mathbb{N}^{*}$,  $\beta_1, \ldots, \beta_m \in \mathbb{R}^+$, and $C\geq0$ such that
$\sum^{m}_{j=1} \beta_j |\evec _j(\xvec)|^2=C$
almost everywhere on 
$E$, then there must hold $C =0$ and $\beta_1= \beta_2 =\ldots =\beta_m = 0$.

\null

\item[ii.]  For every $a\in L^\infty (\Omega; [0,1])$ such that $\int_{\Omega}a(\xvec)\,d\xvec=L|\Omega|$,
one has
\[
\liminf_{j\rightarrow +\infty}\;\frac{e^{2T\lambda_j}-1}{2\lambda_j}  \int_\Omega a(x)|\evec_j(\xvec)|^2\,d\xvec
>\frac{e^{2T\lambda_1}-1}{2\lambda_1}.
\]
\end{itemize}
}

By the analyticity of the eigenfunctions of Stokes system with homogeneous Dirichlet boundary conditions, it is not difficult to show that the first condition holds. 

For the second condition, notice that there exists $\epsilon>0$ and $E \subset \Omega$ of positive measure such that $a\geq \epsilon \chi_{E}$ and 
$$
\int_{\Omega} a(x)|\evec_j(\xvec)|^2\,d\xvec \geq \epsilon \int_{E}|\evec_j(\xvec)|^2\,d\xvec.
$$
From Theorem \ref{4241}, we easily see that 
$$
\liminf_{j\rightarrow +\infty}\;\frac{e^{2T\lambda_j}-1}{2\lambda_j}  \int_\Omega a(x)|\evec_j(\xvec)|^2\,d\xvec = + \infty.
$$

From \cite[Theorem 1]{PTZ}, it follows that problem $(P^T)$ has a unique solution. 

\end{proof}

\begin{remark}
The optimal set given by Theorem \ref{TheoOpt1} is open and semi-analytic\footnote{Here, 
	it is understood that the optimal set is unique up to the set of zero measure.
	A subset of a real analytic finite-dimensional manifold is said to be semi-analytic if it can be written 
	in terms of equalities and inequalities of real analytic functions.}. This  follows from the fact that the eigenfunctions of the Stokes system with homogeneous Dirichlet boundary conditions are analytic. 
\end{remark}

\begin{remark}
A proof of Theorem \ref{TheoOpt1} when $\Omega$ is the unit disk of $\mathbb{R}^2$ can be found in \cite{PTZ}. However, such proof relies on an explicity knowledge of the eigenfunctions of the Stokes operator, which of course can not be extended to general domains. Notice that to prove Theorem \ref{TheoOpt1}, in the general case, the key point is to obtain an observability inequalities with observations over measurable sets of positive measure as in Theorem \ref{4241}.
\end{remark}

\bigskip
% ===========================================================
% NULL CONTROLLABILITY FOR STOKES SYSTEM WITH BOUNDED CONTROLS
% ===========================================================

\subsection{Null controllability for Stokes system with bounded controls}
	Let $\om$ be a non-empty open subset of $\Om$ and consider the following controlled Stokes system
\begin{equation}\label{stokes}
	\left |   
		\begin{array}{lcl}
			\uvec_t - \Delta \uvec  +\nabla p = \vvec\chi_\om &  \mbox{in}&  Q,  \\
			\Div \uvec  = 0 &  \mbox{in}&   Q,  \\
			\uvec = \ovec & \mbox{on}& \Sigma, \\
			\uvec(\cdot,0) = \uvec_0 & \mbox{in}& \Omega.
		\end{array}
	\right. 
\end{equation}

	It is well known that for any $ T > 0$, $\uvec_0\in \Hvec$, and $ \vvec\in \Lvec^2(\omega\times(0,T))$, there 
	exists exactly one solution $(\uvec,p)$ to the Stokes equations~\eqref{stokes} with
$$
	 \uvec\in C^0\left([0, T];\Hvec\right)\cap L^2\left(0, T;\Vvec\right),~p \in L^2(0,T; U),
$$
	where  
\[
	\begin{alignedat}{2}
		&U:=\dis \left\{\psi \in H^1(\Om); \int_\Om \psi(\xvec)\,d\xvec=0\right\}.
    	\end{alignedat}
\]

	In the context of the Stokes system \eqref{stokes}, for $1\leq p\leq \infty$, the {\it $\Lvec^p$- null controllability} problem at time $T$ reads as
	follows:
{\it
\begin{quote}
	 For any $\uvec_0 \in \Hvec$, find a control $ \vvec\in \Lvec^p(\omega\times(0,T))$ such that the associated solution 
	to~\eqref{stokes} satisfies
\end{quote}
\begin{equation}\label{null_condition_stokes}
	\uvec(\xvec,T) = 0\quad\hbox{in}\quad\Omega.
\end{equation}
}

The following result is well-known.
\begin{theorem}\label{th-2}
For any non-empty open subset $\omega$ of  $\Omega$ and any $T>0$, the Stokes system~\eqref{stokes} is $\Lvec^2$-null controllable.
\end{theorem}
	\noindent For the proof, we refer the reader to \cite{CL1,FGIP,FI}.

\

 In practice it would be interesting to take the control steering the solution of the Stokes system to rest to be in $\Lvec^\infty(\omega\times(0,T))$. Nevertheless,  it is not clear how to construct $\Lvec^\infty(\omega\times(0,T))$ controls from $\Lvec^2(\omega\times(0,T))$ controls. Notice that for the case of the heat equation this is always possible since one can use  local regularity results (for more details, see \cite{Manolo}), which is no longer the case for the Stokes system.

From Theorem \ref{obser} we are able to deduce a null controllability for Stokes system with \linebreak $\Lvec^\infty$-controls. 
	More precisely, we have:
\begin{theorem}\label{control_Linfty} For any non-empty open subset $\omega$ of  $\Omega$ and any $T>0$,	the Stokes system~\eqref{stokes} is $\Lvec^\infty$-null-controllable. 
\end{theorem}
\begin{proof}
The proof follows from the duality between observability and controllability and the  $\Lvec^1$-observability inequality \eqref{Obsinequality}.
\end{proof}

	The observability inequality stablished in Theorem \ref{obser} allow us to conclude 
	stronger controllability properties for the Stokes system \eqref{stokes}. In fact  it is possible to control the Stokes system with $\Lvec^\infty$-controls supported in any measurable set of positive measure:
	
\begin{theorem}\label{bang_bang_stokes_0}
 For any $T>0$ and any measurable set of positive measure $\gamma\subset \Omega\times[0,T]$,
	the Stokes system~\eqref{stokes} is $\Lvec^\infty$-null controllable with control supported in $\gamma$. 
\end{theorem}

\bigskip

% ===========================================================
% TIME OPTIMAL CONTROL FOR THE STOKES EQUATIONS
% ===========================================================

\subsection{Time optimal control problem for the Stokes system}
		
Let  $|\cdot|_r: \mathbb{R}^N \rightarrow [0,\infty)$ be the $r$-euclidean norm in $\mathbb{R}^N$, i.e., 
\[
	|\xvec|_r=
	\left \{   
		\begin{array}{lcl}
			(|x_1|^r+\ldots+|x_N|^r)^{1\over r} &  \mbox{if}&  r\in[1,\infty),  \\
			  \noalign{\smallskip}\dis
			\max\{|x_1|,\ldots,|x_N|\} & \mbox{if}& r=\infty,
		\end{array}
	\right.
\]
for every $\xvec \in \mathbb{R}^N$. 

For $r \in [1,\infty]$ fixed and  any $M>0$,  we consider the  {\it set of  admissible controls}
\begin{equation*}%\label{admisible_space}
	\mathcal{U}_{ad}^{M,r}=\{ \vvec\in \Lvec^\infty(\om\times[0,\infty))\,;\, | \vvec(\xvec,t)|_r\leq M 
	\hbox{ \,a.e.  in \,} \om\times[0,\infty)\}
\end{equation*}
and for $\uvec_0\in \Hvec$ given, we define the {\it set of reachable states starting from $\uvec_0$}:
\begin{equation*}%\label{reachable_space}
	\mathcal{R}(\uvec_0,\mathcal{U}_{ad}^{M,r})=\left\{\uvec(\cdot,\tau)\,;\, \tau>0 \hbox{ and }
	 \uvec\hbox{ is the solution of \eqref{stokes} with }  \vvec\in\mathcal{U}_{ad}^{M,r}\right\}.
\end{equation*}
Thanks to Theorem \ref{control_Linfty}, it follows that $\ovec \in \mathcal{R}(\uvec_0,\mathcal{U}_{ad}^{M,r})$, for any $\uvec_0\in\Hvec$.

	In this section, we study the following time optimal 
	control problem: 
\begin{quote}
\it{given $\uvec_0\in\Hvec$ and $\uvec_f\in \mathcal{R}(\uvec_0,\mathcal{U}_{ad}^{M,r})$, find $ \vvec^\star_r\in\mathcal{U}^{M,r}_{ad}$ 
	such that the corresponding solution $\uvec^\star$ of \eqref{stokes} satisfies}
\begin{equation}\label{timeoptimalproblem0}
	\uvec^\star(\tau^\star_r(\uvec_0,\uvec_f))=\uvec_f,
\end{equation}
where $\tau^\star_r(\uvec_0,\uvec_f)$ is the minimal time needed to steer the initial datum $\uvec_0$ to
	the target $\uvec_f$ with controls in $\mathcal{U}_{ad}^{M,r}$, i.e.
\begin{equation}\label{timeoptimalproblem}
	\tau^\star_r(\uvec_0,\uvec_f)=\min_{ \vvec\in \mathcal{U}_{ad}^{M,r}}\left\{\tau\,;\,\uvec(\cdot,\tau)=\uvec_f\right\}.
\end{equation}
\end{quote} 
	
We have the following result:
\begin{theorem}\label{bang_bang_stokes}
	Let $M>0$ and $r\in [1,\infty]$  be given. For every $\uvec_0\in \Hvec$ and any $\uvec_f\in \mathcal{R}(\uvec_0,\mathcal{U}_{ad}^{M,r})$, the time optimal 
	problem \eqref{timeoptimalproblem} has at least one solution. Moreover, any optimal control $\vvec^\star_r$
	satisfies the bang-bang property:  $|\vvec^\star_r(\xvec,t)|_r=M$ for a.e. $(\xvec ,t)\in \omega\times[0,\tau^\star_r(\uvec_0,\uvec_f)]$.
\end{theorem}
\begin{proof}
	
Since $\uvec_f\in \mathcal{R}(\uvec_0,\mathcal{U}_{ad}^{M,r})$, there exists a minimizing sequence
	$(\tau_n, \vvec_n)_{n\geq1}$ such that $ \tau_n \xrightarrow[ n \to \infty ]{}\tau^\star_r(\uvec_0, \uvec_f)$ 
	and $(\vvec_n)_{n\geq1}\subset \mathcal{U}_{ad}^{M,r}$ has the property that the associated solution $\uvec_n$ to \eqref{stokes}
	satisfies $\uvec_n(\cdot,\tau_n)= \uvec_f$ for all $n\geq1$. Also, because $( \vvec_n)_{n\geq1}\subset \mathcal{U}_{ad}^{M,r}$, it follows that $( \vvec_n)_{n\geq1}$ converges 
	weakly-$\star$ to some vector-function $ \vvec^\star\in  \mathcal{U}_{ad}^{M,r}$ in 
	$\Lvec^\infty(\om\times(0,\tau^\star_r(\uvec_0, \uvec_f)))$. 
	
	\null
	
\begin{claim}
$\vvec^\star$ is a solution of the time optimal problem \eqref{timeoptimalproblem0}.
\end{claim} 
\begin{proof}[Proof of the Claim.]
We only have to show that $ \uvec^\star(\cdot,
\tau^\star_r(\uvec_0, \uvec_f) )=\uvec_f$, where  $\uvec^\star$ is the solution of \eqref{stokes} associated
	 to $ \vvec^\star$. 

	%$\widehat \uvec_f=\uvec^\star(\cdot,
	 %\tau^\star_r(\uvec_0, \uvec_f) )$ 
	 
	 To show this, let $\bar \uvec$ be the solution of \eqref{stokes} with $\vvec\equiv\ovec$
	 and $\wvec=\uvec^\star-\bar \uvec$, $\wvec_n=\uvec_n-\bar\uvec$ solutions of 
%\[
%\left |   
%\begin{array}{lcl}
%\yvec_t - \Delta \yvec  +\nabla q = 0 &  \mbox{in}&  Q,  \\
%\Div\yvec  = 0 &  \mbox{in}&   Q,  \\
%\yvec = \mathbf{0} & \mbox{on}& \Sigma, \\
%\yvec(0) = \uvec_0 & \mbox{in}& \Omega,
%\end{array}
%\right. 
%\]	 
	%and the controlled systems
\[
\left |   
\begin{array}{lcl}
\wvec_t - \Delta \wvec  +\nabla \pi =  \vvec^\star1_\om &  \mbox{in}&  Q,  \\
\Div \wvec  = 0 &  \mbox{in}&   Q,  \\
\wvec = \mathbf{0} & \mbox{on}& \Sigma, \\
\wvec_n(0) = \mathbf{0} & \mbox{in}& \Omega,
\end{array}
\right. 
\]
and
\[
\left |   
\begin{array}{lcl}
\wvec_{n,t} - \Delta \wvec_n  +\nabla \pi_n =  \vvec_n1_\om &  \mbox{in}&  Q,  \\
\Div \wvec_n  = 0 &  \mbox{in}&   Q,  \\
\wvec_n = \mathbf{0} & \mbox{on}& \Sigma, \\
\wvec_n(0) = \mathbf{0} & \mbox{in}& \Omega,
\end{array}
\right. 
\]
	respectively.
		
	Now, thanks to the continuity in time of $\bar\uvec$ and that  $\tau_n\xrightarrow[ n \to \infty ]{}\tau^\star(\uvec_0,\uvec_f)$, it follows that 
	$\bar\uvec(\cdot,\tau_n)\xrightarrow[ n \to \infty ]{} \bar\uvec(\cdot,\tau^\star_r(\uvec_0, \uvec_f) )$ in $\Hvec$.  Moreover, it is not difficult to see that
$$
	\langle \wvec_n(\tau_n)-\wvec_n(\tau^\star_r(\uvec_0, \uvec_f)),\varphi \rangle\to 0 \quad \forall \varphi\in \Hvec,
$$
$$
	\langle \wvec_n(\tau^\star_r(\uvec_0, \uvec_f)),\varphi \rangle\to 
	\langle \wvec(\tau^\star_r(\uvec_0, \uvec_f)),\varphi \rangle \quad \forall \varphi\in \Hvec
$$
	and 
$$
	\langle  \wvec_n(\tau_n),\varphi \rangle\to 
	\langle \wvec(\tau^\star_r(\uvec_0, \uvec_f)),\varphi \rangle \quad \forall \varphi\in \Hvec.
$$

\

	Since $\uvec_f=\bar\uvec(\cdot,\tau_n)+\wvec_n(\cdot,\tau_n)$,
	we have that $\langle \uvec_f,\varphi \rangle=\langle \bar\uvec(\cdot,\tau_n)+\wvec_n(\cdot,\tau_n),\varphi \rangle$ for all
	$\varphi\in \Hvec$
	 and
$
	\langle \uvec_f,\varphi \rangle=\langle \bar\uvec(\cdot,\tau^\star_r(\uvec_0, \uvec_f))
	+\wvec(\cdot,\tau^\star_r(\uvec_0, \uvec_f)),\varphi \rangle
		=\langle \uvec^\star(\cdot,
\tau^\star_r(\uvec_0, \uvec_f) ),\varphi \rangle$, for all $ \varphi\in \Hvec$.
\end{proof}

\

	Now, let us show that any optimal control $\vvec^\star\in \mathcal{U}_{ad}^{M,r}$
	satisfies the bang-bang property. To do this, we argue by contradiction.
	
	We consider $\uvec^\star$ the corresponding state (with some pressure) to \eqref{stokes} and suppose that there exist $\eps>0$ and a measurable 
	set of positive measure $\gamma\subset\om\times(0,\tau^\star_r(\uvec_0, \uvec_f))$ such that
\begin{equation}\label{eq:contradic}
	|\vvec^\star(\xvec,t)|_r<M-\eps \quad ((\xvec,t)\in \gamma).
\end{equation}

Choosing $\delta_0>0$ small enough such that
\[
\left \{   
\begin{array}{l}
	\tau_0=\tau^\star_r(\uvec_0, \uvec_f)-\delta_0>0,\\
	\hbox{the set } \Gamma=\{(\xvec,t)\in \om\times(0,\tau_0)\,:\, 
	(\xvec,t)\in \gamma \}
	\hbox{ has positive measure},
\end{array}
\right. 
\]
and using the time continuity of  $\uvec^\star$,  there exists $\delta\in (0,\delta_0)$
	such that 
\begin{equation}\label{eq:sol}
	\|\uvec_0-\uvec^\star(\cdot,\delta)\|_\Hvec\leq {\eps\over C_{obs}(\tau_0,\Gamma)},
\end{equation}
	where  $C_{obs}(\tau_0,\Gamma)$ is the observability constant given by Theorem \ref{obser} for the control domain $\Gamma$ 
	at time $\tau_0$. 

	From Theorem \ref{bang_bang_stokes_0},  there exists a control 
	$ \vvec\in \Lvec^\infty(\om\times (0,\tau_0))$ with
\[
\left \{   
\begin{array}{l}
	\text{ supp}\,\vvec\subset \Gamma,\\
	\hbox{the associated solution }  \uvec\hbox{ satisfies } 
	\uvec(\cdot,0)=\uvec_0-\uvec^\star(\cdot,\delta) \hbox{ and } \uvec(\cdot,\tau_0)=\mathbf{0},\\
	%\hbox{where } \yvec(t+\delta_0)=\tilde\yvec(t) \hbox{ is solution for the Stokes equation with }
	% v=0 \hbox{ and }\tilde \yvec(0)= \uvec_0-\uvec^\star(\delta).\\
	 \| \vvec\|_{\Lvec^\infty (\Gamma)}\leq C_{obs}(\tau_0,\Gamma)\|\uvec_0-\uvec^\star(\delta)\|_\Hvec.
\end{array}
\right. 
\]

	Thus, from \eqref{eq:sol} we have that
\[
	 \| \vvec\|_{\Lvec^\infty (\om\times(0,\tau_0))}\leq \eps.
\]

	Now, let $\widehat  \vvec\in \Lvec^\infty (\om\times(0,\tau_0))$ be defined by
$$
	\widehat  \vvec(x,t)= \vvec^\star(\xvec,t+\delta) + \vvec(\xvec,t) \quad (t\in [0,\tau_0]).
$$

	Noticing that $\tau_0+\delta\leq \tau^\star_r(\uvec_0, \uvec_f)$, using the fact that
	$\text{supp}\,\vvec \subset\Gamma$ and estimate \eqref{eq:contradic}, it follows that $\widehat \vvec\in\mathcal{U}_{ad}^{M,r}$.

	Finally, setting $\widehat \uvec(\xvec,t)=\uvec^\star(\xvec,t+\delta)+\uvec(\xvec,t)$
	and $\widehat p(\xvec,t)=p^\star(\xvec,t+\delta)+p(\xvec,t)$,  we have that $\widehat \uvec(\cdot,0)=\uvec_0$,
	$\widehat \uvec(\tau^\star_r(\uvec_0, \uvec_f)-\delta)=\uvec_f$ and that
\[
\widehat\uvec_t - \Delta \widehat \uvec +\nabla\widehat p = \widehat \vvec1_\om.
\]

	Hence, $\widehat \vvec\in \mathcal{U}_{ad}^{M,r}$ is a control which steers $\uvec_0$ to $\uvec_f$ at time
	$\tau^\star_r(\uvec_0, \uvec_f)-\delta$. This contradicts the definition of 
	$\tau^\star_r(\uvec_0, \uvec_f)$ and then the bang-bang property holds.
\end{proof}

	About the uniqueness of the optimal control for problem \eqref{timeoptimalproblem}, we have the following result:
\begin{proposition}
	Let $M>0$ and $r\in(1,\infty)$. For any $\uvec_0\in \Hvec$ and every $\uvec_f\in \mathcal{R}(\uvec_0,\mathcal{U}_{ad}^{M,r})$, the time optimal control
	problem \eqref{timeoptimalproblem0}-\eqref{timeoptimalproblem}  has a unique solution $ \vvec^\star_r$ 
	which satisfies a bang-bang property: $|\vvec^\star_r(\xvec,t)|_r=M$ for a.e. $(\xvec ,t)\in \omega\times[0,\tau^\star_r(\uvec_0,\uvec_f)]$.
\end{proposition}
\begin{proof}
	The existence of solution and the bang-bang property is a consequence of Theorem \ref{bang_bang_stokes}. 
	We only have to prove the uniqueness of solution. Thus, let $ \vvec$ and $\hvec$ be two time optimal controls in $\mathcal{U}^{M,r}_{ad}$.
	Thanks to the linearity, $\wvec={1\over2}( \vvec+\hvec)$ is also a time optimal control. From
	Theorem \ref{bang_bang_stokes}, $\wvec$ also satisfies the bang-bang property. Therefore, we have that $|\vvec(\xvec,t)|_r=|\hvec(\xvec,t)|_r=|\wvec(\xvec,t)|_r=M$, a.e.
 	in $\om\times(0,\tau^\star_r(\uvec_0,\uvec_1))$. Now, if $\vvec(\xvec,t)\neq\hvec(\xvec,t)$ 
	in a measurable set of positive measure $\mathcal{D}\subset \om\times (0,\tau^\star_r(\uvec_0,\uvec_1))$,
	then, thanks to the fact that any norm $|\cdot|_r$ for $r\in(1,\infty)$ is {\it uniformly convex} in $\mathbb{R}^N$, 
	we have that $|\wvec(\xvec,t)|_r<M$ a.e. in  $\mathcal{D}\subset \om \times (0,\tau^\star_r(\uvec_0,\uvec_1))$. 
	This contradicts the bang-bang property for $\wvec$.
\end{proof}

% ===========================================================
% APPENDIX
% ===========================================================

\appendix

% ===========================================================
% REAL-ANALYTIC ESTIMATES FOR SOLUTIONS TO THE POISSON EQUATION
% ===========================================================

\section{Real-analytic estimates for solutions to the Poisson equation}\label{appendix}

	In this appendix we prove the following lemma which was used in the proof of Theorem \ref{4241}.
\begin{lemma}\label{analyticityDelta}
	Assume that $f$ is an real-analytic function in $B_{R}(\xvec_0)$ verifying
\begin{equation}\label{423-6}
	|\partial_\xvec^{\alpha}f(\xvec)| \leq {M|\alpha|!\over (R\rho_0)^{|\alpha|}}
	\;\;\;\;\text{for all}\;\; \xvec\in B_R(\xvec_0)\;\;\text{and}\;\;\alpha \in \mathbb{N}^N,
\end{equation}
	with some positive constants $M$ and $\rho_0$. Let $u\in L^2(B_R(\xvec_0))$ satisfying 
	the Poisson equation
\begin{equation}\label{Poisson}
			- \Delta u  = f\;\;\text{in}\;\;B_R(\xvec_0).
\end{equation}
	Then, $u$ is real-analytic in  $B_{R/2}(\xvec_0)$ and has the estimate
\begin{align}\label{estA0}
	\| \partial_\xvec^{\alpha} u\|_{L^{\infty}(B_{R/2}(\xvec_0))}
	\leq  {|\alpha|!\over (R\tilde\rho)^{|\alpha|+1}} \bigl(\|u\|_{L^2(B_R(\xvec_0))} + M \bigl),
	\;\;\text{for all}\;\;\alpha\in \mathbb N^N,
\end{align}
	 where $\tilde\rho$ is a  constant depending only on the dimension $N$ and $\rho_0$.
\end{lemma}
	A proof of the lemma \ref{analyticityDelta} for $f\equiv 0$ can be found in \cite{MorreyNirenberg}. 
	For the sake of completeness, we give a proof for the non-homogeneous case. 
\begin{proof}
	By rescaling, it suffices to prove the estimate \eqref{estA0} when $R=1$ and $\xvec_0=\ovec$.
	
	Since $f$ is real-analytic in $B_1(\ovec)$, by the interior regularity for solutions of elliptic equations,  
	we have that $u$ is smooth in  $B_{1}(\ovec)$.  Hence, we have that 
$$
	-\Delta \partial_\xvec^{\alpha} u(\xvec)= \partial_\xvec^{\alpha} f(\xvec)\;\;\;\text{for all}\;\;\xvec\in B_1(\ovec),
$$
	for every $\alpha =(\alpha_1, \ldots, \alpha_N)\in\mathbb N^N$.
 
	Multiplying the above equation by $(1-|\xvec|^2)^{2(|\alpha|+1)}\partial_\xvec^{\alpha}u$ gives 
\begin{equation}\label{eq:alpha_mult}
	-(1-|\xvec|^2)^{2(|\alpha|+1)}\partial_\xvec^{\alpha} u (\xvec)\Delta \partial_\xvec^{\alpha}u(\xvec)
	=(1-|\xvec|^2)^{2(|\alpha|+1)}\partial_\xvec^{\alpha} u(\xvec) \partial_\xvec^{\alpha} f(\xvec),
	\;\;\;\;\forall\,\xvec\in B_1(\ovec),
\end{equation}
	and integration by parts gives
%\[
%\begin{alignedat}{2}
%	-\iint_{B_1(\ovec)}(1-|\xvec|^2)^{2(|\alpha|+1)}\partial_\xvec^{\alpha} u\Delta \partial_\xvec^{\alpha} u\,d\xvec 
%	= &~\iint_{B_1(\ovec)}(1-|\xvec|^2)^{2(|\alpha|+1)}|\nabla \partial_\xvec^{\alpha}u|^2\,d\xvec\\
%	&-4(|\alpha|+1)\iint_{B_1(\ovec)}  (1-|\xvec|^2)^{2|\alpha|+1}(\nabla\partial_\xvec^{\alpha} u\cdot \xvec)
%	\partial_\xvec^{\alpha} u\,d\xvec.
%\end{alignedat}
%\]
%	In this way, from the above inequality and \eqref{eq:alpha_mult}, we deduce
\[
\begin{alignedat}{2}
	\iint_{B_1(\ovec)} (1-|\xvec|^2)^{2(|\alpha|+1)}| \nabla \partial_\xvec^{\alpha}u|^2\,d\xvec
	=&~4(|\alpha|+1)\iint_{B_1(\ovec)}  (1-|\xvec|^2)^{2|\alpha|+1}  
	(\nabla \partial_\xvec^{\alpha} u\cdot \xvec)\partial_\xvec^{\alpha} u\, d\xvec\\
	&+\iint_{B_1(\ovec)} (1-|\xvec|^2)^{2(|\alpha|+1)}\partial_\xvec^{\alpha}u\partial_\xvec^{\alpha}f\,d\xvec.
\end{alignedat}
\]
	
	Now, thanks to the Young's inequality, we have the following estimate
%\[
%\begin{alignedat}{2}
% 	\iint_{B_1(\ovec)}\!\!\! (1-|\xvec|^2)^{2(|\alpha|+1)}| \nabla \partial_\xvec^{\alpha} u|^2\,d\xvec
%	\leq&~4(|\alpha|+1)\left({1\over8(|\alpha|+1)}\iint_{B_1(\ovec)}\!\!\!(1-|\xvec|^2)^{2(|\alpha|+1)}
%	|\nabla \partial_\xvec^{\alpha} u|^2\,d\xvec\right. \\
% 	&\left.+2(|\alpha|+1)\iint_{B_1(\ovec)}(1-|\xvec|^2)^{2|\alpha|}
%	|\xvec|^2|\partial_\xvec^{\alpha} u|^2\,d\xvec \right)\\
%	&+{1\over2}\iint_{B_1(\ovec)}(1-|\xvec|^2)^{2|\alpha|}  |\partial_\xvec^{\alpha} u|^2\,d\xvec\\ 
%	&+ {1\over2}\iint_{B_1(\ovec)}(1-|\xvec|^2)^{2(|\alpha|+2)} |\partial_\xvec^{\alpha} f|^2d\xvec.
%\end{alignedat}
%\]
%	which gives
\[
\begin{alignedat}{2}
	\iint_{B_1(\ovec)} (1-|\xvec|^2)^{2(|\alpha|+1)}| \nabla \partial_\xvec^{\alpha} u|^2d\xvec 
	\leq&~[16(|\alpha|+1)^2+1] \iint_{B_1(\ovec)}(1-|\xvec|^2)^{2|\alpha|}|\partial_\xvec^{\alpha}u|^2\,d\xvec\\
	&+\iint_{B_1(\ovec)}|\partial_\xvec^{\alpha}f|^2\,d\xvec.\\
\end{alignedat}
\]

	Since $f$ satisfies \eqref{423-6}, we get
\[
\begin{alignedat}{2}
 	\iint_{B_1(\ovec)} (1-|\xvec|^2)^{2(|\alpha|+1)}|\nabla \partial_\xvec^{\alpha} u|^2\,d\xvec
	\leq&~17(|\alpha|+1)^2\iint_{B_1(\ovec)}  (1-|\xvec|^2)^{2|\alpha|}|\partial_\xvec^{\alpha} u|^2\, d\xvec\\
	&+|B_1(0)|\left|M|\alpha|!\over\rho_0^{|\alpha|}\right|^2.
\end{alignedat}
\]
	Therefore, we obtain
\begin{equation}\label{estA1}
	\left\|(1-|\xvec|^2)^{|\alpha|+1} \nabla \partial_\xvec^{\alpha} u\right\|_{L^2(B_1(\ovec))} 
	\leq 5\left[(|\alpha|+1)\left\|(1-|\xvec|^2)^{|\alpha|}\partial_\xvec^{\alpha} u\right\|_{L^2(B_1(\ovec))}
	 +{M|\alpha|!\over\rho_0^{|\alpha|}} \right],
\end{equation}
	for every $\alpha =(\alpha_1, \ldots, \alpha_N)\in \mathbb{N}^N$. In particular, taking $\alpha =(0, \ldots, 0)$, 
	we deduce the estimate 
\[
	\left\|(1-|\xvec|^2) \nabla u\right\|_{L^2(B_1(\ovec))}\leq 5\left(\|u\|_{L^2(B_1(\ovec))} + M \right).
\]

%	Next, we prove the following claim by induction.
%\bigskip	
%	

	By induction, we have the inequality
%\begin{claim}  {\it  the inequality
\begin{align}\label{claim}
	\big\|(1-|\xvec|^2)^{|\alpha|} \partial_\xvec^{\alpha} u\big\|_{L^2(B_1(\ovec))}
	\leq \rho^{-|\alpha| -1}|\alpha|! \left(\|u\|_{L^2(B_1(\ovec))} + M \right),
\end{align}
	for some constant $0 <\rho < \min \left\{ \rho_0, {1\over6}\right\}$ and every $\alpha =(\alpha_1, \ldots, \alpha_N)\in \mathbb{N}^N$.%}
	
	It is not difficult to see that estimate \eqref{claim} leads to \eqref{estA0}.

\end{proof}

\bigskip

 \noindent\textbf{Acknowledgements}.
The authors would like to appreciate Prof. E. Tr\'elat and Prof. G. Lebeau  for the stimulating conversations during 
this work.

\end{document}